\documentclass[final,leqno,onefignum,onetabnum]{siamltex1213}

\usepackage{amssymb,amsmath}
\usepackage{tikz}
\usepackage{tikz-qtree}
\usepackage{pgfplotstable}
\usepackage{mathtools}
\mathtoolsset{showonlyrefs=true}
\DeclarePairedDelimiter\paren{\lparen}{\rparen}

\newtheorem{remark}{Remark}

\newenvironment{tikztree}{
    \tikzpicture[grow'=up,sibling distance=0.08cm, level distance=0.2cm, baseline=-0.02cm,
    edge from parent/.style={draw, edge from parent path={(\tikzparentnode) -- (\tikzchildnode)}}]
    \tikzstyle{every node}=[draw, circle, fill=black, inner sep=1pt, minimum size=1mm]
}{\endtikzpicture}

\newenvironment{tikzftree}{
    \tikzpicture[grow=right,sibling distance=0.08cm, level distance=0.2cm, baseline=-0.1cm,
    edge from parent/.style={draw, edge from parent path={(\tikzparentnode) -- (\tikzchildnode)}}]
    \tikzstyle{every node}=[draw, circle, fill=black, inner sep=1pt, minimum size=1mm]
}{\endtikzpicture}

\newcommand{\ti}{\begin{tikztree}
\node {};
\useasboundingbox (0.1,0);
\end{tikztree}}

\newcommand{\tii}{\begin{tikztree}
\Tree[.{} [.{} ] ];
\useasboundingbox (0.1,0);
\end{tikztree}}

\newcommand{\tiii}{\begin{tikztree}[baseline]
\Tree[.{} [.{} ][.{} ] ];
\end{tikztree}}

\newcommand{\tiv}{\begin{tikztree}[baseline]
\Tree[.{} [.{} [.{} ]] ];
\useasboundingbox (0.1,0);
\end{tikztree}}

\newcommand{\tv}{\begin{tikztree}[baseline]
\Tree[.{} [.{} ] [.{} ] [.{} ] ];
\end{tikztree}}

\newcommand{\tvi}{\begin{tikztree}[baseline]
\Tree[.{} [.{} ] [.{} [.{} ]] ];
\end{tikztree}}

\newcommand{\tvii}{\begin{tikztree}[baseline]
\Tree[.{} [.{} [.{} ] [.{} ]] ];
\end{tikztree}}

\newcommand{\tviii}{\begin{tikztree}[baseline]
\Tree[.{} [.{} [.{} [.{} ]]] ];
\useasboundingbox (0.1,0);
\end{tikztree}}

\newcommand{\tix}{\begin{tikztree}[baseline]
\Tree[.{} [.{} ] [.{} ] [.{} ] [.{} ]];
\end{tikztree}}

\newcommand{\tx}{\begin{tikztree}[baseline]
\Tree[.{} [.{} ] [.{} ] [.{} [.{} ]] ];
\end{tikztree}}

\newcommand{\txi}{\begin{tikztree}[baseline]
\Tree[.{} [.{} ] [.{} [.{} ] [.{} ]] ];
\end{tikztree}}

\newcommand{\txii}{\begin{tikztree}[baseline]
\Tree[.{} [.{} ] [.{}  [.{}  [.{} ]]] ];
\end{tikztree}}

\newcommand{\txiii}{\begin{tikztree}[baseline]
\Tree[.{} [.{} [.{} ]] [.{} [.{} ]] ];
\end{tikztree}}

\newcommand{\txiv}{\begin{tikztree}[baseline]
\Tree[.{} [.{} [.{} ] [.{} ] [.{} ]]  ];
\end{tikztree}}

\newcommand{\txv}{\begin{tikztree}[baseline]
\Tree[.{} [.{} [.{} ] [.{} [.{} ]]]  ];
\end{tikztree}}

\newcommand{\txvi}{\begin{tikztree}[baseline]
\Tree[.{} [.{} [.{}  [.{} ] [.{} ]]] ];
\end{tikztree}}

\newcommand{\txvii}{\begin{tikztree}[baseline]
\Tree[.{} [.{} [.{}  [.{}  [.{} ]]]] ];
\end{tikztree}}

\newcommand{\fti}{\begin{tikzftree}
\Tree[.{} [.{} [.{} [.{} ] [.{} ]] ]];
\end{tikzftree}}

\newcommand{\calT}{\mathcal T}
\newcommand{\calF}{\mathcal F}
\newcommand{\calFT}{\mathcal{FT}}
\newcommand{\calO}{\mathcal O}

\newcommand{\rmd}{\mathrm d}
\newcommand{\bbR}{\mathbb R}

\makeatletter
\def\diag{\mathop{\operator@font diag}\nolimits}
\makeatother

\title{A characterization of energy-preserving methods and the construction
of parallel integrators for Hamiltonian systems} 

\author{
Yuto Miyatake\thanks{Department of Computational Science and Engineering,
Graduate School of Engineering, Nagoya University, Furo-cho, Chikusa-ku,
Nagoya 464-8603, Japan
(\email{miyatake@na.nuap.nagoya-u.ac.jp}).}
\and
John~C. Butcher\thanks{Department of Mathematics, University of Auckland,
Auckland 1142, New Zealand
(\email{butcher@math.auckland.ac.nz}).}
}

\begin{document}
\maketitle
\slugger{mms}{xxxx}{xx}{x}{x--x}

\begin{abstract}
High order energy-preserving methods for Hamiltonian systems are presented.
For this aim, an energy-preserving condition of continuous stage Runge--Kutta methods
is proved.
Order conditions are simplified and parallelizable conditions are also given.
The computational cost of our high order methods is comparable to
that of the average vector field method of order two.
\end{abstract}



\pagestyle{myheadings}
\thispagestyle{plain}
\markboth{}{}

\section{Introduction}
This paper is
concerned with the numerical integration of
Hamiltonian systems  
of the form
\begin{align}
\frac{\rmd}{\rmd t}y = f(y),\qquad y(t_0)=y_0\in\mathbb{R}^N,
\end{align}
where $f(y) = S\nabla H(y)$
with a non-singular skew-symmetric constant matrix $S$
and sufficiently differentiable Hamiltonian $H:\mathbb{R}^N\to\mathbb{R}$.
Although the dimension $N$ of the system is even, in the case of
Hamiltonian systems, we do not assume this because the results
are also applicable to Poisson systems
with constant Poisson structure, where $S$ is not necessarily non-singular and
$N$ need not be even.

Several geometric properties are known in classical mechanics. 
For example,
the exact flow is energy-preserving in the sense that
\begin{align}
\frac{\rmd}{\rmd t}H(y) = \nabla H(y)^\top \frac{\rmd y}{\rmd t}
= \nabla H(y) ^\top S \nabla H(y) = 0,
\end{align}
as well as being symplectic.
Numerical integrators exactly inheriting such properties 
are called geometric numerical integrators~\cite{ha06},
and the most remarkable advantage of adopting geometric numerical integrators
is that they often give qualitatively correct numerical solutions 
over an extremely long period of time.
Although 
integrators inheriting such geometric properties as much as possible
would be preferable,
it is known to be impossible to construct 
methods inheriting both symplecticity and energy-preservation~\cite{cfm06,zh88}.
This is because if such methods exist the numerical flow coincides with the exact flow.
In the last decades,
methods inheriting one of these properties, i.e. symplectic methods
and energy-preserving methods, have been developed. 

Although both symplectic methods and energy-preserving methods
have their own advantages,
much more attention has been devoted to symplectic methods.
One of the most remarkable advantages of symplectic methods
is that
they also nearly preserve the energy without any drift.
It should  also be
noted that
there is a symplectic condition for Runge--Kutta (RK) methods~\cite{la88,sa88,su88}:
\begin{align*}
b_ia_{ij}+b_ja_{ji}=b_ib_j, \qquad 1\leq i,j\leq s,
\end{align*}
where $s$ is the stage number, and $a_{ij}$ and $b_i$ are the coefficients of RK methods.
This condition is not only sufficient but also necessary under a certain assumption.
In fact, symplectic implicit RK methods exist whose typical examples are
Gauss methods.
For more details, we refer the reader to~\cite{ha06}.

While symplectic methods are useful for most Hamiltonian systems,
it is sometimes mandatory to adopt energy-preserving methods.
For example, errors in the Hamiltonian obtained by symplectic methods
can sometimes be fatal for chaotic dynamics. 
Thus, it is strongly hoped that the study on energy-preserving methods
reaches to the same maturity as symplectic methods.
However, the study on energy-preserving methods is more challenging
and, in fact, newer than that on symplectic methods.
One reason is that no RK method is energy-preserving
for general Hamiltonian systems~\cite{ce09}.

The projection method is a relatively simple method of
constructing energy-preserving integrators.
The idea is very simple, but the implementation and
the construction of the projection is specific to individual problems.
An even more serious disadvantage of projection-based integrators is that
the long term behaviour  usually deteriorates.  Compared with the projection method,
the discrete gradient method, which was first proposed by Gonzalez~\cite{go96}
(see also McLachlan et al.~\cite{mc99}),
is rather a more sophisticated and systematic approach.

After some pioneering studies,
substantial progress was made in the last decade.
From a theoretical viewpoint,
Faou--Hairer--Pham~\cite{fhp04} and Chartier--Faou--Murua~\cite{cfm06} 
revealed that
there exist energy-preserving B-series integrators
other than the exact flow
by showing an energy-preserving condition.
Furthermore,
concrete energy-preserving methods have also been developed,
Quispel--McLaren proposed the average vector field (AVF) method~\cite{qu08}.
This method is energy-preserving, a subclass of the discrete gradient method,
a B-series method
and of order two.
A more recent interest is to find a high order energy-preserving method
as an extension of the AVF method.
Hairer proposed the AVF collocation method~\cite{ha10},
and Brugnano et al. proposed the Hamiltonian boundary value method~\cite{bru10}.
These methods are essentially the same methods,
although the latter method has been developed mainly for polynomial 
Hamiltonian systems.
These methods are based on so called continuous stage RK (CSRK) methods.
It should be noted that
roughly speaking
the computational cost of the AVF collocation method
is almost the same as the same order Gauss method. 

The main aim of this paper is to construct a more efficient 
energy-preserving method than exists at present.
To achieve this aim we focus our attention on a matrix $M$ which
characterizes CSRK methods. 
In terms of $M$,  we consider (i) an energy preserving condition,
(ii) order conditions and (iii) criteria for parallel implementation using real arithmetic.
We then put together these three perspectives to construct high order
energy-preserving methods which can be computed in parallel.
We then derive concrete integrators of up to order six.

The paper is organized as follows.
In \autoref{sec2}, some existing energy-preserving methods are briefly reviewed
with the introduction of CSRK methods.
In \autoref{sec3}, 
an energy-preserving condition of CSRK methods is discussed.
In addition to the sufficient condition already given in \cite{mi14}, 
we show that the condition is also necessary under a small assumption.
The condition is expressed in terms of the matrix $M$.
In \autoref{sec4},
order conditions of energy-preserving CSRK methods are characterized
in terms of the matrix, where simplifying assumptions play an important role.
In \autoref{sec5},
parallel energy-preserving methods are constructed.
We first consider a parallelizable condition of CSRK methods,
and then construct parallel energy-preserving methods.
Furthermore, we investigate fourth-order methods in detail.
The derived fourth order methods can be implemented faster
than the same order AVF collocation method,
but an actual error of our methods per each time step becomes bigger
than that of the existing method.
We therefore have to discuss to what extent our method is efficient.
Finally, concluding remarks are given in \autoref{sec6}.

\section{Energy-preserving continuous stage Runge--Kutta methods}
\label{sec2}
In this section,
we review some existing energy-preserving methods
and continuous stage Runge--Kutta methods.

The discrete gradient method proposed by Gonzalez \cite{go96}
(see also McLachlan et al.~\cite{mc99})
is a first systematic method for deriving energy-preserving integrators.
The key idea of the method is to introduce a sophisticated approximation of
the gradient.
We first define a discrete gradient,
a map $\overline{\nabla}H: \bbR^N\times\bbR^N \to \bbR^N$, satisfying
\begin{align}
& H(x) - H(y) = \overline{\nabla}H(x,y)^\top (x-y), \label{dg1} \\
& \overline{\nabla} H(x,x) = H(x) \label{dg2}
\end{align}
for any $x,y\in\bbR^N$,
and then define an integrator by
\begin{align}\label{dgmethod}
\frac{y_1-y_0}{h} = S \overline{\nabla} H(y_1,y_0),
\end{align}
where $h$ denotes the stepsize.
Here, the condition \eqref{dg1} corresponding to the chain rule
$\frac{\rmd}{\rmd t} H(x) = \nabla H(x)^\top \dot{x}$
is called the discrete chain rule,
and the condition \eqref{dg2} requires the consistency.
Due to the discrete chain rule and skew-symmetry of the matrix $S$,
the energy-preservation for the discrete gradient method \eqref{dgmethod}
can be easily verified:
\begin{align*}
\frac{1}{h}\paren*{H(y_1)-H(y_0)} 
= \overline{\nabla}H(y_1,y_0)^\top \frac{y_1-y_0}{h}
= \overline{\nabla}H(y_1,y_0)^\top S \overline{\nabla}H(y_1,y_0) = 0.
\end{align*}
The discrete gradient is not generally unique
and several approaches have been proposed.
The average vector field (AVF) gradient proposed by Quispel--McLaren~\cite{qu08}
\begin{align*}
\overline{\nabla}H(y_1,y_0) = \int_0^1 \nabla H(\tau y_0 + (1-\tau)y_1) \, \rmd \tau
\end{align*}
is one of several possible approaches.
In general, the AVF method reads
\begin{align}\label{avf}
y_1 = y_0 + h\int_0^1 f(\tau y_0 + (1-\tau)y_1) \, \rmd \tau.
\end{align}
The AVF method is a symmetric, second order and furthermore B-series method
as will be shown later.

After the proposition of the AVF method,
the AVF method was extended to higher orders.
Hairer extended the AVF method
by slightly modifying the idea of the standard collocation methods~\cite{ha10}
(see also the Hamiltonian boundary value method~\cite{bru10}).
In this paper, we call the extended method the AVF collocation method.
As is the case with the standard collocation method which can be interpreted as
a RK method,
the AVF collocation method can be interpreted as
a continuous stage Runge--Kutta (CSRK) method.

\begin{definition}[Continuous stage Runge--Kutta methods]
Let $A_{\tau,\zeta}$ be a polynomial in $\tau$ and $\zeta$.
Assume that $A_{0,\zeta}=0$.
We denote by $s$ the polynomial degree of $A_{\tau,\zeta}$ in $\tau$.
Let $B_\zeta$ be defined by $B_\zeta = A_{1,\zeta}$.
We search for an $s$-degree polynomial $Y_\tau$
($\tau\in[0,1]$) and $y_1$
such that they satisfy
\begin{align}
Y_\tau &= y_0 + h\int_0^1 A_{\tau,\zeta} f(Y_\zeta )\, \rmd \zeta, \label{CSRK1}\\
y_1 &= y_0 + h\int_0^1 B_\tau f(Y_\tau ) \,\rmd \tau . \label{CSRK2}
\end{align}
A one-step method $y_0\mapsto y_1$ is called
an $s$-degree continuous stage Runge--Kutta (CSRK) method. 
A CSRK method is said to be consistent if $\int_0^1 B_\zeta \,\rmd\zeta=1$.
\end{definition}

\begin{remark}
CSRK methods were originally due to Butcher~\cite{bu72} (see also~\cite{bu08}),
where $A_{\tau,\zeta}$ was not assumed to be a polynomial but to be bounded.
In this paper, we focus on a polynomial $A_{\tau,\zeta}$
because this restriction simplifies discussion below and more general functions
do not seem to be of greater benefit to our aim than polynomial functions.
\end{remark}

In the above definition,
$A_{0,\zeta}=0$ is assumed so that $Y_0=y_0$.
$Y_\tau$ can be seen as an approximation of $y(t_0+C_\tau h)$,
where $C_\tau = \int_0^1 A_{\tau,\zeta}\,\rmd\zeta$.
The final stage $y_1$ is also expressed by $y_1 = Y_1$ due to
the relation $B_\zeta = A_{1,\zeta}$.
Since in this paper we shall focus on $A_{\tau,\zeta}$
which is polynomial of degree $s$ in $\tau$ and $s-1$ in $\zeta$,
we often express $A_{\tau,\zeta}$ as
\begin{align}
A_{\tau,\zeta} =
\begin{bmatrix}
\tau & \frac{\tau^2}{2} & \cdots & \frac{\tau^s}{s}
\end{bmatrix}
M
\begin{bmatrix}
1 \\ \zeta \\ \vdots \\ \zeta^{s-1}
\end{bmatrix} \label{csrkm}
\end{align}
with a constant matrix $M\in\mathbb{R}^{s\times s}$.

In the case of the AVF collocation method,
the polynomial $A_{\tau,\zeta}$ is given by
\begin{align}
A_{\tau,\zeta} = \sum_{i=1}^s \frac{1}{b_i} \int_0^\tau L_i(\alpha)\,\rmd \alpha
L_i(\zeta),
\end{align}
where
\begin{align}
L_i (\tau) = \prod_{j=1,\ j\neq i}^s \frac{\tau-c_j}{c_i-c_j}, \qquad
b_i = \int_0^1 L_i (\tau) \,\rmd \tau
\end{align}
and $c_1,\dots,c_s$ are distinct, real numbers.
It was proved in~\cite{ha10} that the AVF collocation method is energy-preserving
for Hamiltonian systems
independently of the $c_i$ values.
If the $c_i$ values are the zeros of the $s$-th shifted Legendre polynomial,
then the method has order $p=2s$.
Here lists concrete expressions of $A_{\tau,\zeta}$ for $s=1,2,3$:
\begin{alignat}{2}
s=1: &\qquad A_{\tau,\zeta}= \tau, \\
s=2: &\qquad  A_{\tau,\zeta} = \tau \paren*{(4-3\tau)-6(1-\tau)\zeta}, \\
s=3: &\qquad  A_{\tau,\zeta} =
\tau \paren*{(9-18\tau +10\tau^2) - 12(3-8\tau+5\tau^2)\zeta 
+ 30(1-3\tau+2\tau^2)\zeta^2 }.
\end{alignat}
Note that in the case of $s=1$ and $A_{\tau,\zeta}=\tau$,
the corresponding CSRK method coincides with the AVF method \eqref{avf}.

The following theorem ensures the unique solvability of CSRK methods.

\begin{theorem}[Unique solvability~\cite{ta12}]
Suppose that $f$ satisfies a Lipschitz condition
\begin{align*}
\| f(\eta) - f(\overline{\eta}) \| \leq L\| \eta - \overline{\eta}\|
\end{align*}
with a constant $L$.
If 
\begin{align*}
h < \frac{1}{L\paren*{\max_{\tau\in [0,1]}\int_0^1|A_{\tau,\zeta}|\,\rmd\zeta} } ,
\end{align*}
there exists a unique solution of \eqref{CSRK1}.
\end{theorem}

This theorem can be proved in a similar way to the proof of the unique solvability of RK methods
(see, e.g.~\cite{bu08,hw96}) as mentioned in~\cite{ta12}.
For the convenience of the reader, this proof is included below.

\begin{proof}
First, we note that the function space
of $s$-degree polynomials with the property $Y_0=y_0$
can be identified with $\bbR^{sN}$.
We define a metric on $\bbR^{sN}$ by
\begin{align*}
\rho (Y_\tau, \overline{Y}_\tau ) = \max_{\tau\in[0,1]} \| Y_\tau - \overline{Y}_\tau \|,
\end{align*}
and consider the map
\begin{align*}
Y_\tau \mapsto \phi (Y_\tau) = y_0 + h\int_0^1 A_{\tau,\zeta} f(Y_\zeta)\,\rmd\zeta.
\end{align*}
Then we estimate $\rho (\phi (Y_\tau),\phi (\overline{Y}_\tau))$ as follows:
\begin{align*}
\rho (\phi (Y_\tau),\phi (\overline{Y}_\tau))
&=
h\max_{\tau\in[0,1]} \left\| \int_0^1 A_{\tau,\zeta}
\paren*{f(Y_\zeta)-f(\overline{Y}_\zeta)}\,\rmd\zeta \right\| \\
&\leq
h\max_{\tau\in[0,1]}  \int_0^1 |A_{\tau,\zeta}|
\left\| f(Y_\zeta)-f(\overline{Y}_\zeta)\right\|\,\rmd\zeta \\
&\leq
hL\max_{\tau\in[0,1]}  \int_0^1 |A_{\tau,\zeta}|
\left\| Y_\zeta-\overline{Y}_\zeta\right\|\,\rmd\zeta \\
&\leq
hL \paren*{\max_{\tau\in[0,1]}  \int_0^1 |A_{\tau,\zeta}|\,\rmd\zeta}
\max_{\zeta\in[0,1]} \left\| Y_\zeta-\overline{Y}_\zeta\right\| \\
&=
hL \paren*{\max_{\tau\in[0,1]}  \int_0^1 |A_{\tau,\zeta}|\,\rmd\zeta}
\rho (Y_\tau,\overline{Y}_\tau ).
\end{align*}
If the stepsize $h$ satisfies $hL \paren*{\max_{\tau\in[0,1]}  \int_0^1 |A_{\tau,\zeta}|\,\rmd\zeta}<1$,
the contraction mapping theorem ensures
the existence and uniqueness of the fixed point of the map $\phi$.
\end{proof}

As pointed out in \cite{ha10},
the numerical solution of a CSRK method
can be interpreted as a B-series.
Below we introduce some notation to be needed in the subsequent sections.
See~\cite{bu08,ha06} for more details on B-series.

Let $\calT$ be the set of rooted trees
\begin{align}
\calT=\{\ti, \tii, \tiii, \tiv, \tv, \tvi, \tvii, \tviii, \dots  \}.
\end{align}
We denote by $|t|$ the number of vertices of a tree $t$.
A new tree obtained by connecting the roots of $t_1,\dots,t_m\in\calT$
is denoted by $t=[t_1,\dots,t_m]$.
If some of $t_1,\dots,t_m$ equal to each other,
we write $t=[t_1^{r_1},\dots,t_m^{r_m}]$,
e.g. $[\ti,\tii,\tii] = [\ti , \tii^2]$.
The symmetry coefficient $\sigma:\calT\to\bbR$ is defined recursively by
\begin{align*}
\sigma (\ti) = 1, \qquad
\sigma (t) = \prod_{k=1}^m r_k ! \sigma (t_k)^{r_k},
\end{align*}
where $t=[t_1^{r_1},\dots,t_m^{r_m}]$.
The elementary differential is a mapping $F(t):\bbR^N \to \bbR^N$,
defined recursively by
\begin{align*}
F(\ti) (y) = f(y), \qquad
F(t) (y) = f^{(m)} \paren*{F(t_1)(y),\dots,F(t_m)(y)}.
\end{align*}
For a map $a:\calT\cup\{ \emptyset\} \to\bbR$, a formal series of the form
\begin{align*}
B(a,y) = a(\emptyset) y + \sum_{t\in \calT}\frac{h^{|t |}}{\sigma (t)}
a(t)F(t)(y)
\end{align*}
is called a B-series.
The exact time-$h$ flow of a differential equation $\dot{y}=f(y)$
can be expressed as $y(t_0+h) = B(e,y(t_0))$,
where the coefficients $e$ are given by
\begin{align*}
e(\emptyset) = e(\ti) = 1, \qquad
e(t) = \frac{1}{|t|} e(t_1)\cdots e(t_m).
\end{align*}
As is the case with Runge--Kutta methods,
every CSRK method is a B-series integrator
$y_1 = B(\phi,y_0)$,
where the coefficients $\phi$, called the elementary weights,
are given by $\phi(\emptyset)=\phi(\ti)=1$, $\phi_\tau (\ti)=C_\tau$
and
\begin{align*}
\phi_\tau (t) = \int_0^1 A_{\tau,\zeta}
\phi_\zeta (t_1)\cdots \phi_\zeta (t_m) \,\rmd\zeta, \qquad
\phi (t) = \int_0^1 B_\tau \phi_\tau(t_1)\cdots \phi_\tau (t_m)\,\rmd\tau.
\end{align*}
Here are the first four elementary weights:
\begin{alignat}{2}
\phi (\ti ) &= \int_0^1 B_\tau \,\rmd\tau,
& \qquad \phi (\tii ) &= \int_0^1 B_\tau A_{\tau,\zeta} \,\rmd\tau\rmd\zeta, \\ 
\phi (\tiii ) &= \int_0^1 B_\tau A_{\tau,\zeta} A_{\tau,\kappa} \,
\rmd\tau\rmd\zeta\rmd\kappa,
& \qquad \phi (\tiv ) &= \int_0^1 B_\tau A_{\tau,\zeta} A_{\zeta,\kappa} \,
\rmd\tau\rmd\zeta\rmd\kappa.
\end{alignat}

In addition to the set of trees $\calT$,
we define a set of forests $\calF$ and a set of free trees $\calFT$.
The forest is an unordered finite collection of trees from $\calT$:
\begin{align*}
\calF=\{
\ti, \ti \ti, \tii, \ti \ti \ti, \ti \tii, \tiii, \tiv, \dots
\}.
\end{align*}
An element of $t\in\cal F$ is written as $t = t_1^{r_1}\cdots t_p^{r_p}$ 
($t_1,\dots,t_p\in \calT$),
and the sum of the vertices are denoted by $|t|$, 
e.g. $|\ti\tii^2| = |\ti \tii\tii|=5$.
For $t\in \calT$, $B_-(t)$ denote the forest consisting of the subtrees of $t$,
e.g. $B_-(\tx) = \ti^2\tii$.
For a forest $t = t_1^{r_1}\cdots t_p^{r_p}$,
we define a map $a:\calF \to \bbR$ by
$a(t) = a(t_1)^{r_1} \cdots a(t_p)^{r_p}$.

Next, we introduce the set of free trees $\calFT$.
Note that there is a group of trees whose elements have topologically the same shapes,
e.g. the shapes of $\tx$, $\txi$, $\txv$ and $\txvi$ are the same.
We define the set $\calFT$
by using the Butcher product, 
for two trees $u=[u_1,\dots,u_q]\in\calT$ and $v\in\calT$,
\begin{align*}
u\circ v := [u_1,\dots,u_q,v] \in \calT.
\end{align*}
Note that $u\circ v \neq v \circ u$ in general.
A tree $u\circ v$ can be obtained by shifting the root of $v\circ u$.
Multiple shift of root induces an equivalence.
All trees in the same equivalent class has the same number of vertices,
and have topologically the same shape.
We call each equivalent tree a free tree
and denote the set of free trees by $\cal{FT}$.
The canonical projection is denoted by $\pi: \calT \to \calFT$.
This map is not bijective,
e.g. $\pi^{-1}(\fti) = \{ \tx,\txi,\txv,\txvi \} $.
For two equivalent trees $u$ and $v$, we define a distance $\kappa (u,v)$
by the minimum number of root shifts necessary to obtain $u$ from $v$.
For example, $\kappa (u,u)=0$, $\kappa(u,v)=\kappa(v,u)$ and $\kappa (\txv,\txvi) = 3$.
Any free tree containing a member with a factorisation of the form $u\circ u$
is called superfluous tree.

\section{Necessary and Sufficient condition for energy-preservation}
\label{sec3}
In order to construct efficient energy-preserving CSRK methods,
it is mandatory to characterize an energy-preserving condition for CSRK methods.
The sufficient condition was already given in \cite{mi14}.
Although one can proceed to the subsequent sections by using the condition,
we also show that the sufficient condition is also necessary 
under a small assumption. 
The necessity indicates that it is almost impossible to find
energy-preserving CSRK methods outside the framework presented in this paper.

\subsection{Sufficient condition}

\begin{theorem}[\cite{mi14}] \label{suffcond}
A CSRK method is energy-preserving, if
the matrix $M$ in \eqref{csrkm} is symmetric, i.e.
if $\frac{\partial}{\partial\tau}A_{\tau,\zeta}$
is symmetric.
\end{theorem}

For readers' convenience, we here show the proof of this theorem.
\begin{proof}
If the matrix $M$ is symmetric,
it can be diagonalized by an orthogonal matrix $P$: $M = P^\top \Lambda P$,
where $\Lambda$ is a diagonal matrix 
whose elements are given by $\Lambda_{ii}=\lambda_i$.
The partial derivative $\frac{\partial}{\partial\tau}A_{\tau,\zeta}$ 
is expressed as
\begin{align*}
\frac{\partial}{\partial \tau} A_{\tau,\zeta}
=
\sum_{i=1}^s \lambda_i \paren*{P\varphi (\tau)}_i \paren*{P\varphi (\zeta)}_i ,
\end{align*}
where $\varphi (\tau) = [1,\tau,\dots,\tau^{s-1}]^\top$ and $(\cdot)_i$
denotes the $i$-th component.
By using these notation,
we have
\begin{align*}
H(y_1)-H(y_0)
&=
\int_0^1 \frac{\rmd}{\rmd \tau} H(Y_\tau)\,\rmd\tau
=\int_0^1 \dot{Y}_\tau^\top \nabla H (Y_\tau)\,\rmd\tau \\
&=
h\int_0^1 \paren*{\int_0^1 \frac{\partial}{\partial \tau}
A_{\tau,\zeta} S\nabla H(Y_\zeta)\,\rmd\zeta} ^\top \nabla H (Y_\tau)\,\rmd\tau \\
&=
h\sum_{i=1}^s \lambda_i
\paren*{\int_0^1 (P\varphi (\zeta))_i \nabla H(Y_\zeta)\,\rmd\zeta}^\top S^\top
\paren*{\int_0^1 (P\varphi (\tau))_i \nabla H(Y_\tau)\,\rmd\tau} \\
&=0.
\end{align*}
The last equality follows from the skew-symmetry of $S$.
\end{proof}

\begin{remark}
Tang and Sun have also provided essentially the same sufficient condition in a different manner~\cite{ta14}. 
\end{remark}

\subsection{Necessary condition}
The idea of proving the necessity is to use
a necessary and sufficient condition of energy-preserving B-series
methods shown by Chartier et al.~\cite{cfm06}.
We will also use some lemmas which are motivated by Celledoni et al.~\cite{cos14}.

The necessary and sufficient condition of energy-preserving B-series methods
is summarized in the following theorem.
\begin{theorem}[\cite{cfm06}]
A B-series $B(a,\cdot)$ is energy-preserving if and only if
\begin{align}\label{eq:cfmcond}
\sum_{u\in\pi^{-1}(\bar{t})}
\frac{(-1)^{\kappa(t,u)}}{\sigma(u)}
a(B_-(u))=0,
\qquad
\text{for all }
\bar{t}\in\calFT,
\end{align}
where $t$ is a member of $\pi^{-1}(\bar{t})$.
\end{theorem}

As corollaries, we get the following properties.
\begin{lemma}\label{th:bccond}
If a consistent B-series method is energy-preserving,
then it follows that
\begin{align} \label{eq:t2}
a([\ti^{k-1}]) \paren*{=\int_0^1 B_\tau C_\tau^{k-1}\,\rmd \tau}
= \frac{1}{k}\qquad \text{for all } k\in\mathbb{N}^+ (:= \{ 1,2,3,\dots \}).
\end{align} 
\end{lemma} 

\begin{proof}
The idea of the proof is similar to that of Theorem 2.2 in \cite{cos14}.
This theorem is proved by considering a free tree
such that $k$ vertices are connected to a common vertex,
e.g. $\tix$ for $k=4$.
Any free tree of this type has two equivalent trees: $t_1 = [\ti^k]$ and $t_2=[[\ti^{k-1}]]$.
The symmetry coefficients are easily calculated to be
$\sigma (t_1) = k! $ and $\sigma(t_2) = (k-1)!$.
Note that $a(B_- (t_1)) = a(\ti)^k = 1$
and $a(B_-(t_2)) = a([\ti^{k-1}])$.
Substituting $\sigma$ and $a$ into \eqref{eq:cfmcond},
we obtain the condition \eqref{eq:t2}.
\end{proof}

\begin{lemma}
If a consistent CSRK method is energy-preserving,
then it follows that
\begin{align}\label{eq:coscond1}
&p\int_0^1\int_0^1
B_\tau C_\tau ^{p-1} A_{\tau,\zeta} C_\zeta ^{q}
\,\rmd \tau \rmd \zeta
-
q\int_0^1\int_0^1
B_\tau C_\tau ^{q-1} A_{\tau,\zeta} C_\zeta ^{p}
\,\rmd \tau \rmd \zeta \\
&\hspace{15em}=
\frac{1}{q+1} - \frac{1}{p+1},
\qquad \text{for all } p,q\in\mathbb{N}^+. \nonumber
\end{align}
\end{lemma}

\begin{proof}
The idea of the proof is similar to that of Lemma 2.3 in \cite{cos14}.
To prove this condition, we introduce double bush trees.
As illustrated in the picture below,
a double bush tree $t_{p,q}$ is a free tree having
$p$ leaves on the left side and $r$ leaves on the right side
(this figure is an example with $p=2$ and $q=3$).

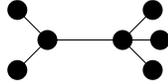
\begin{center}
\begin{tikzpicture}[every node/.style={circle,fill=black,scale=0.8}]
    \coordinate (L) at (0cm, 0cm);
    \coordinate (R) at (1cm, 0cm);
    \coordinate (L1) at (-0.4cm, -0.4cm);
    \coordinate (L2) at (-0.4cm, 0.4cm);
    \coordinate (R1) at (1.4cm, -0.4cm);
    \coordinate (R2) at (1.5cm, 0cm);
    \coordinate (R3) at (1.4cm, 0.4cm);
    \node at (L) {}; \node at (R) {}; \node at (L1) {}; \node at (L2) {};
    \node at (R1) {}; \node at (R2) {};\node at (R3) {}; 
    \draw[-] (L) -- (R);
    \draw[-] (L1) -- (L);
    \draw[-] (L2) -- (L);
    \draw[-] (R1) -- (R);
    \draw[-] (R2) -- (R);
    \draw[-] (R3) -- (R);
\end{tikzpicture}
\end{center}
It is clear that $t_{p,p}$ is superfluous, and $t_{p,q}=t_{q,p}$.

There are four distinct but equivalent rooted trees for $t_{p,q}$:
$t_1=[[\ti^{p-1}[\ti^{q}]]]$,
$t_2=[\ti^p[\ti^q]]$,
$t_3=[[\ti^p]\ti^q]$ and
$t_4=[[[\ti^p]\ti^{q-1}]]$
(it is clear that $\kappa (t_i,t_j) = |i-j|$).
Here are examples for $t_{2,3}$:
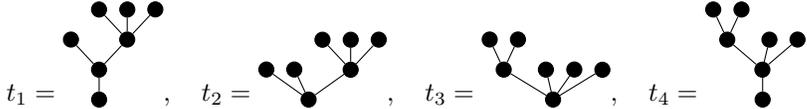
\begin{align*}
t_1=
\begin{tikzpicture}[grow'=up,sibling distance=0.16cm, level distance=0.4cm, baseline=-0.02cm,
    edge from parent/.style={draw, edge from parent path={(\tikzparentnode) -- (\tikzchildnode)}}]
    \tikzstyle{every node}=[draw, circle, fill=black, inner sep=1pt, minimum size=2mm]
    \Tree[.{} [.{} [.{} ] [.{} [.{} ] [.{} ] [.{} ]]]  ];
\end{tikzpicture},
\quad
t_2 =
\begin{tikzpicture}[grow'=up,sibling distance=0.16cm, level distance=0.4cm, baseline=-0.02cm,
    edge from parent/.style={draw, edge from parent path={(\tikzparentnode) -- (\tikzchildnode)}}]
    \tikzstyle{every node}=[draw, circle, fill=black, inner sep=1pt, minimum size=2mm]
    \Tree[.{} [.{} ] [.{} ] [.{} [.{} ] [.{} ] [.{} ]] ];
\end{tikzpicture},
\quad
t_3 =
\begin{tikzpicture}[grow'=up,sibling distance=0.16cm, level distance=0.4cm, baseline=-0.02cm,
    edge from parent/.style={draw, edge from parent path={(\tikzparentnode) -- (\tikzchildnode)}}]
    \tikzstyle{every node}=[draw, circle, fill=black, inner sep=1pt, minimum size=2mm]
    \Tree[.{} [.{} [.{} ] [.{} ]] [.{} ] [.{} ] [.{} ] ];
\end{tikzpicture},
\quad
t_4 = 
\begin{tikzpicture}[grow'=up,sibling distance=0.16cm, level distance=0.4cm, baseline=-0.02cm,
    edge from parent/.style={draw, edge from parent path={(\tikzparentnode) -- (\tikzchildnode)}}]
    \tikzstyle{every node}=[draw, circle, fill=black, inner sep=1pt, minimum size=2mm]
    \Tree[.{} [.{} [.{} [.{} ] [.{} ]] [.{} ] [.{} ]]];
\end{tikzpicture}.
\end{align*}
Thus, it follows that
\begin{align*}
& a(B_-(t_1)) = a([\ti^{p-1},[\ti^q]]) = \int_0^1 \int_0^1
B_\tau C_\tau ^{p-1} A_{\tau,\zeta} C_\zeta ^{q} \, \rmd \tau \rmd \zeta, \\
& a(B_-(t_2)) = a(\ti)^p a([\ti^q]) = \frac{1}{q+1}, \\
& a(B_-(t_3)) = a(\ti)^q a([\ti^p]) = \frac{1}{p+1}, \\
& a(B_-(t_4)) = a([\ti^{q-1},[\ti^p]]) = \int_0^1 \int_0^1
B_\tau C_\tau ^{q-1} A_{\tau,\zeta} C_\zeta ^{p} \, \rmd \tau \rmd \zeta.
\end{align*}
The second and third are due to \autoref{th:bccond}.
The symmetry coefficients are calculated to be
\begin{align*}
\sigma (t_1) = (p-1)! q!, \quad
\sigma (t_2) =  \sigma (t_3) = p!q!, \quad
\sigma (t_4) = p! (q-1) !.
\end{align*}
Substituting $\kappa$, $a$ and $\sigma$ into \eqref{eq:cfmcond},
we obtain the condition \eqref{eq:coscond1}.
\end{proof}

We consider an $s\times \infty$ matrix $\Phi $ defined by
\begin{align*}
\Phi_{i,j} = \int _0^1 \tau^{i-1} C_\tau^j\,\rmd \tau.
\end{align*}
We now show the following theorem.
\begin{theorem}\label{maintheorem}
If a consistent CSRK method for which $\Phi$ is of full rank $s$
is energy-preserving,
then $M=M^\top$.
\end{theorem}

\begin{proof}
We see by contradiction that
the consistency and the full rank assumption indicate $B_\tau = C_\tau^\prime$.
Under the assumption $B_\tau = C_\tau^\prime$,
we can rewrite the condition \eqref{eq:coscond1} as follows.
Using the integration-by-parts formula, we have
\begin{align*}
&\int_0^1 \int_0^1 B_\tau C_\tau ^{p-1} A_{\tau,\zeta} C_\zeta^q \, \rmd \tau \rmd \zeta \\
&= 
\int _0^1 \left[ \frac{1}{p} C_\tau ^{p} A_{\tau,\zeta} C_\zeta^q \right]_0^1
\rmd \zeta
-
\frac{1}{p} \int_0^1 \int_0^1 C_\tau^p A_{\tau,\zeta}^\prime C_\zeta^q \, 
\rmd\tau\rmd\zeta \qquad (A_{\tau,\zeta} ^\prime = \frac{\partial}{\partial \tau} A_{\tau,\zeta}) \\
&=
\frac{1}{p} \int _0^1 B_\zeta C_\zeta^q \,\rmd \zeta -
\frac{1}{p} \int_0^1 \int_0^1 C_\tau^p A_{\tau,\zeta}^\prime C_\zeta^q \, \rmd\tau\rmd\zeta \\
&= 
\frac{1}{p(q+1)} - \frac{1}{p} \int_0^1 \int_0^1 C_\tau^p
A_{\tau,\zeta}^\prime C_\zeta^q \, \rmd\tau\rmd\zeta .
\end{align*}
Similarly,
\begin{align*}
\int_0^1 \int_0^1 B_\tau C_\tau ^{q-1} A_{\tau,\zeta} C_\zeta^p \, \rmd \tau \rmd \zeta
=
\frac{1}{q(p+1)} - \frac{1}{q} \int_0^1 \int_0^1 C_\tau^q
A_{\tau,\zeta}^\prime C_\zeta^p \, \rmd\tau\rmd\zeta .
\end{align*}
Substituting these relations into the condition \eqref{eq:coscond1},
we have
\begin{align*}
\int_0^1 \int_0^1 C_\tau^p A_{\tau,\zeta}^\prime C_\zeta^q \,\rmd\tau\rmd\zeta
=
\int_0^1 \int_0^1 C_\tau^q A_{\tau,\zeta}^\prime C_\zeta^p \,\rmd\tau\rmd\zeta .
\end{align*}
Note that
\begin{align*}
C_\tau^p A_{\tau,\zeta}^\prime C_\zeta^q 
=
C_\tau^p [1,\tau , \dots ,\tau ^{s-1}] M 
\begin{bmatrix}
1 \\ \zeta \\ \vdots \\ \zeta^{s-1}
\end{bmatrix} C_\zeta^q
\end{align*}
and 
\begin{align*}
\int_0^1 \int_0^1 C_\tau^q A_{\tau,\zeta}^\prime C_\zeta^p \,\rmd\tau\rmd\zeta 
& =
\int_0^1 \int_0^1
C_\tau^q [1,\tau , \dots ,\tau ^{s-1}] M 
\begin{bmatrix}
1 \\ \zeta \\ \vdots \\ \zeta^{s-1}
\end{bmatrix} C_\zeta^p \,\rmd\tau\rmd\zeta \\
& =
\int_0^1 \int_0^1
C_\zeta^q [1,\zeta , \dots ,\zeta ^{s-1}] M 
\begin{bmatrix}
1 \\ \tau \\ \vdots \\ \tau^{s-1}
\end{bmatrix} C_\tau^p \,\rmd\tau\rmd\zeta \\
& =
\int_0^1 \int_0^1
C_\tau^p [1,\tau , \dots ,\tau ^{s-1}] M ^\top
\begin{bmatrix}
1 \\ \zeta \\ \vdots \\ \zeta^{s-1}
\end{bmatrix} C_\zeta^q \,\rmd\tau\rmd\zeta.
\end{align*}
Therefore, we have
\begin{align} 
\paren*{\int_0^1 C_\tau^p [1,\tau , \dots ,\tau ^{s-1}] \,\rmd \tau}
(M-M^\top)
\paren*{\int_0^1 C_\zeta^q \begin{bmatrix}
1 \\ \zeta \\ \vdots \\ \zeta^{s-1}
\end{bmatrix} \,\rmd \zeta}
=0, \label{eq:pmp}
\end{align}
which is equivalent to
\begin{align*}
\Phi^\top (M-M^\top) \Phi = 0.
\end{align*}
Since we assume that
the matrix $\Phi$ is of full rank $s$,
it follows that $M=M^\top$. 
\end{proof}

\begin{remark}
\label{rem:con}
Let us discuss the assumption about the rank of $\Phi$.
The technique of the proof of \autoref{maintheorem}
is motivated by the proof of the necessary condition of symplectic RK methods,
especially Theorem~7.10 in \cite[Chapter VI]{ha06},
where a similar $s\times\infty$ matrix characterizes the reducibility of RK methods.
In contrast to RK methods, however, 
we believe that for any consistent CSRK methods the rank of the matrix $\Phi$
is always of full rank $s$.
This conjecture can be easily verified for monotonic $C_\tau$
or the case the function space spanned by
$\{ C_\tau^i\}_{i\in \mathbb{N}^+}$ is dense in $\{v \ | \ v\in C[0,1], \ v(0)=0 \}$,
but
the conjecture for more general cases has not been proved.
We leave the proof of this conjecture to our future work.
\end{remark}

\section{Families of energy-preserving continuous stage Runge--Kutta methods and order conditions}
\label{sec4}

In this section, we characterize order conditions of energy-preserving CSRK methods
in terms of the matrix $M$ \eqref{csrkm},
based on the simplifying assumptions
\begin{alignat*}{4}
B(\rho):& \qquad& \int_0^1 B_\tau C_\tau^{k-1}\, \rmd\tau &= \frac{1}{k}, && \qquad k=1,\dots,\rho,\\
C(\eta):& \qquad& \int_0^1 A_{\tau,\zeta}C_\zeta^{k-1}\,\rmd\zeta 
&= \frac{C_\tau^k}{k}, && \qquad k=1,\dots,\eta, \\
D(\xi) :& \qquad& \int_0^1 B_\tau C_\tau^{k-1}A_{\tau,\zeta} \,\rmd \tau
&= \frac{B_\zeta}{k}\paren*{1-C_\zeta^k}, && \qquad k=1,\dots,\xi,
\end{alignat*}
and construct several families of energy-preserving methods.

The order of a CSRK method satisfying the simplifying assumption
$B(\rho)$, $C(\eta)$ and $D(\xi)$
is at least $\min (\rho,2\eta+2,\eta+\xi+1)$.
The following two theorems can further simplify the order conditions.

\begin{theorem}
For any consistent CSRK methods with the property $M=M^\top$,
the $B(\rho)$ condition is satisfied for all $\rho=1,2,\dots$.
\end{theorem}

\begin{proof}
For any CSRK methods with the property $M=M^\top$, it follows that$B_\tau=C_\tau^\prime$.
The consistency indicates that 
$C_1 = \int_0^1 B_\tau \,\rmd\tau=1$.
Hence, it follows that for any $k$
\begin{align*}
\int_0^1 B_\tau C_\tau^{k-1}\,\rmd\tau =
\frac{1}{k}\left[ C_\tau^k \right]_0^1 = \frac{1}{k}.
\end{align*} 
\end{proof}

\begin{theorem}
For any consistent CSRK methods with the property $M=M^\top$,
the $C(\eta+1)$ condition is equivalent to the $D(\eta)$ condition.
\end{theorem}

\begin{proof}
We prove
\begin{align}
\int_0^1 A_{\tau,\zeta} C_\zeta^k \,\rmd\zeta= \frac{C_\tau^{k+1}}{k+1}
 \label{eq:Ck1}
\end{align}
is equivalent to
\begin{align}
\int_0^1 B_\tau C_\tau^{k-1} A_{\tau,\zeta} \,\rmd\tau
=\frac{B_\zeta}{k} \paren*{1-C_\zeta^k}.
\label{eq:Dk}
\end{align}

Due to $M=M^\top$, it follows that $C_\tau^\prime=B_\tau$.
Recall that we always assume that $C_0 = \int_0^1 A_{0,\zeta}\,\rmd\zeta=0$,
and the consistency indicates $C_1=1$.
Let
\begin{align*}
\tilde{A}_{\tau,\zeta} = 
\begin{bmatrix}
\tau & \frac{\tau^2}{2} & \dots & \frac{\tau^s}{s}
\end{bmatrix}
M
\begin{bmatrix}
\zeta \\ \zeta^2/2 \\ \vdots \\ \zeta^s/s
\end{bmatrix},
\end{align*}
so that $\frac{\partial}{\partial \zeta} \tilde{A}_{\tau,\zeta} = A_{\tau,\zeta}$
and
$\tilde{A}_{\tau,\zeta} = \tilde{A}_{\zeta,\tau}$.

We start with \eqref{eq:Ck1}.
By using the integration-by-parts formula, we see that \eqref{eq:Ck1} is equivalent to
\begin{align*}
\left[ \tilde{A}_{\tau,\zeta} C_\zeta^k \right]_0^1
-\int_0^1 \tilde{A}_{\tau,\zeta} k C_\zeta^{k-1} B_\zeta\,\rmd\zeta
=
\frac{C_\tau^{k+1}}{k+1}.
\end{align*}
Here, we note that $\left[ \tilde{A}_{\tau,\zeta} C_\zeta^k \right]_0^1
= \tilde{A}_{\tau,1}=\tilde{A}_{1,\tau}$.
By changing $\tau$ and $\zeta$ each other, we obtain
\begin{align}
\tilde{A}_{1,\zeta} - k \int_0^1 B_\tau C_\tau^{k-1} \tilde{A}_{\tau,\zeta}\,\rmd\tau
=
\frac{C_\zeta^{k+1}}{k+1}.
\label{eq:A2}
\end{align}
Differentiating the both sides with respect to $\zeta$ leads to
\begin{align}
B_\zeta - k \int_0^1 B_\tau C_\tau^{k-1} A_{\tau,\zeta}\,\rmd\tau
=
C_\zeta^k B_\zeta, 
\label{eq:A1}
\end{align}
which is equivalent to \eqref{eq:Dk}.
Note that \eqref{eq:A2} is equivalent to \eqref{eq:A1} due to $C_0=0$.
\end{proof}

In the following subsections,
we characterize the order conditions of energy-preserving CSRK methods
in terms of the matrix $M$.

\subsection{Energy-preserving integrators with $B_\zeta=1$}
First, we consider the simplest case $B_\zeta =1$.

\begin{theorem}
\label{th:order2k}
A CSRK method is energy-preserving and of order at least $p=2\eta$
if the symmetric matrix $M \in \bbR ^{s\times s}$ satisfies
\begin{align}
\begin{bmatrix}
\frac{1}{k} & \frac{1}{k+1} & \dots & \frac{1}{k+s-1} 
\end{bmatrix}
M = i_k^\top , \qquad k = 1,\dots,\eta,
\label{eq:cond2k}
\end{align}
where all components of $i_k\in\bbR^s$ are zero except for the $k$-th component
which is $1$.
\end{theorem}

\begin{proof}
We obtain $B_\zeta=1$ by substituting $k=1$ into \eqref{eq:cond2k},
and this indicates that $C_\tau = \tau$.
We consider an integrator satisfying the $C(\eta)$ condition.
Since 
\begin{align*}
\int_0^1 A_{\tau,\zeta}C_\zeta^{k-1}\,\rmd\zeta
&=
\int_0^1
\begin{bmatrix}
\tau & \frac{\tau^2}{2}& \cdots & \frac{\tau^s}{s}
\end{bmatrix}
M
\begin{bmatrix}
\zeta^{k-1} \\ \vdots \\ \zeta^{s+k-2}
\end{bmatrix}\,\rmd\zeta \\
&=
\begin{bmatrix}
\tau & \frac{\tau^2}{2}& \cdots & \frac{\tau^s}{s}
\end{bmatrix}
M
\begin{bmatrix}
1/k \\ \vdots \\ 1/(s+k-1)
\end{bmatrix},
\end{align*}
this coincides with $C_\tau^k/k = \tau^k/k$ for any $\tau$
if and only if 
\begin{align}
\begin{bmatrix}
\frac{1}{k} & \frac{1}{k+1} & \dots & \frac{1}{k+s-1} 
\end{bmatrix}
M = i_k^\top.
\end{align}
\end{proof}

\begin{remark}
For the condition \eqref{eq:cond2k},
it is possible to choose $\eta = s$.
In this case, \eqref{eq:cond2k} is equivalent to
\begin{align*}\def\arraystretch{1.3}
\begin{bmatrix}
1 & \frac{1}{2} & \cdots & \frac{1}{s} \\
\frac{1}{2} & \frac{1}{3} & \cdots & \frac{1}{s+1}\\
\vdots & \vdots & \ddots & \vdots \\
\frac{1}{s} & \frac{1}{s+1} & \cdots & \frac{1}{2s-1}
\end{bmatrix}
M
= I_s,
\end{align*}
and thus $M$ can be written as $M = H^{-1}$,
where $H$ is the Hilbert matrix defined by
$H_{ij} = 1/(i+j-1)$.
In this case, the CSRK method coincides with the AVF collocation method
of order $2s$.
\end{remark}

\subsection{Energy-preserving integrators with $B_\zeta\neq 1$}
It is also possible to derive energy-preserving integrators
which do not satisfy $B_\zeta=1$.
We here illustrate a derivation of 4-degree fourth order integrators.

Let $s=4$ and write $\displaystyle{\begin{bmatrix}
1 & \frac{1}{2} & \frac{1}{3} & \frac{1}{4}
\end{bmatrix}M = v^\top}$.
Then $B_\zeta$ and $C_\tau$ are expressed as
\begin{align*}
B_\zeta = v^\top \begin{bmatrix}
1 \\ \zeta \\ \zeta ^2 \\ \zeta ^3
\end{bmatrix},
\qquad
C_\tau = \begin{bmatrix}
\tau & \frac{\tau^2}{2} & \frac{\tau^3}{3} & \frac{\tau^4}{4}
\end{bmatrix} v.
\end{align*}
If the $C(2)$ assumption is satisfied, we have
\begin{align} \label{bnc2}
\int_0^1 A_{\tau,\zeta} C_\zeta \,\rmd\zeta = \frac{C_\tau^2}{2}.
\end{align}
This relation indicates that the polynomial degree of $C_\tau$
is at most 2, because the polynomial degree in $\tau$
of the left hand side is 4.
We write
\begin{align}
v^\top = \begin{bmatrix}
r & 2 (1-r) & 0 & 0
\end{bmatrix}, \qquad r\neq 1
\end{align}
among several possibilities for which
$\int_0^1 B_\zeta\,\rmd\zeta = 1$.
It follows from \eqref{bnc2} that
\begin{align}
\begin{bmatrix}
\frac{r+2}{6} & \frac{r+3}{12} & \frac{r+4}{20} & \frac{r+5}{30}
\end{bmatrix}
M=
\begin{bmatrix}
0 & r^2 & 3r(1-r) &2(1-r)^2
\end{bmatrix}.
\end{align}
Hence, if a symmetric matrix $M$ satisfies 
\begin{align}\def\arraystretch{1.3}
\begin{bmatrix}
1 & \frac{1}{2} & \frac{1}{3} & \frac{1}{4} \\
\frac{r+2}{6} & \frac{r+3}{12} & \frac{r+4}{20} & \frac{r+5}{30}
\end{bmatrix}
M = 
\begin{bmatrix}
r & 2(1-r) & 0 & 0 \\
0 & r^2 & 3r(1-r) &2(1-r)^2
\end{bmatrix},
\end{align}
the method has order 4.
As an example, if we select $r=0$, this characterization reduces to
\begin{align}\def\arraystretch{1.3}
\begin{bmatrix}
1 & \frac{1}{2} & \frac{1}{3} & \frac{1}{4} \\
\frac{1}{3} & \frac{1}{4} & \frac{1}{5} & \frac{1}{6}
\end{bmatrix}
M = 
\begin{bmatrix}
0 & 2 & 0 & 0 \\
0 & 0 & 0 & 2
\end{bmatrix}.
\end{align}
In this case, $M$ can be expressed with three free parameters 
$\alpha$, $\beta$ and $\gamma$:
\begin{align}\def\arraystretch{1.3}
\begin{bmatrix}
\alpha & \frac{1}{2} & \beta & \frac{1}{6}\\
\frac{1}{2}& \frac{1}{4} & \frac{1}{6} & \frac{1}{8} \\
\beta & \frac{1}{6} & \gamma & \frac{1}{10} \\
\frac{1}{6} & \frac{1}{8} & \frac{1}{10} & \frac{1}{12}
\end{bmatrix}
M = 
\begin{bmatrix}
1 & & & \\
& 1 & & \\
& & 1 & \\
& & & 1
\end{bmatrix}.
\end{align}
It seems that $(\alpha, \beta , \gamma) = (1,\frac{1}{4}, \frac{1}{8})$
are the simplest choices, and in this case $M$ is calculated to be
\begin{align}\def\arraystretch{1.3}
M=
\begin{bmatrix}
-\frac{6}{5} & \frac{72}{5} & -36 & 24\\
\frac{72}{5} & -\frac{144}{5} & -48 & 72 \\
-36 & -48 & 720 & -720 \\
24 & 72 & -720 & 720 
\end{bmatrix}.
\end{align}

As illustrated above, energy-preserving integrators can be obtained 
even for the cases $B_\zeta\neq 1$.
However, such cases require larger degrees,
and thus are less practical than the cases $B_\zeta=1$.

\section{Parallel energy-preserving methods}
\label{sec5}
In this section, we construct new energy-preserving methods
which can be implemented more efficiently than
the AVF collocation method.
In the context of RK methods,
it is known that
the computational cost of solving implicit RK methods
can be reduced if a RK matrix $A$ has only real, distinct eigenvalues
(see, e.g.~\cite{bu13,hw96ii}).
The key of the construction of new integrators
is to apply similar idea to CSRK methods
and put it together with the energy-preserving condition and order conditions.

\subsection{Implementation of Runge--Kutta methods with real eigenvalues}
In general, in order to proceed with a RK process,
we have to solve a system of nonlinear equations of size $sN$ per each time step.
Here we explain that
if the matrix $A$ has only real, distinct eigenvalues, the computational costs
can be reduced, especially in a parallel architecture.

We consider solving an $s$-stage implicit RK method
\begin{align*}
Y_i &= y_0 + h\sum_{j=1}^s a_{ij} f(Y_j),\qquad i=1,\dots,s, \\
y_1 &= y_0 + h\sum_{i=1}^s b_i f(Y_i).
\end{align*}
In order to compute the internal stages $Y_1,\dots,Y_s \in \mathbb{R}^N$,
we have to solve a system of nonlinear equations
\begin{align}\label{eq:RK:ns}
\Phi (Y) = Y- e_s \otimes y_0 - h (A\otimes I_N) F  = 0 
\end{align}
where $\otimes$ denotes the Kronecker product,
$I_N\in \bbR^{N\times N}$ is an identity matrix and
\begin{align}
Y = \begin{bmatrix}
Y_1\\ \vdots \\ Y_s
\end{bmatrix} \in \bbR^{sN}, \quad
e_s =\begin{bmatrix}
1\\ \vdots \\ 1
\end{bmatrix} \in \bbR^s, \quad
F = \begin{bmatrix}
f(Y_1) \\ \vdots \\ f(Y_s)
\end{bmatrix} \in \bbR^{sN}.
\end{align}
The size of the system \eqref{eq:RK:ns} is $sN$.
We usually apply the simplified Newton method to solve the nonlinear system \eqref{eq:RK:ns},
and then obtain the iteration formula
\begin{align}\label{sNewton}
\paren*{I_{sN} - h A \otimes J_0 } \rho^l = - \Phi (Y^l), \qquad
Y^{l+1} = Y^l + \rho^l , \qquad l =0,1,2,\dots ,
\end{align}
where $J_0$ denotes the Jacobian matrix, i.e. $J_0 = S \nabla ^2 H(y_0)$ for Hamiltonian systems.

Below, we show that the linear system \eqref{sNewton} of size $sN$
can be computed efficiently
if the matrix $A$ has only real, distinct eigenvalues.

If all eigenvalues of $A$ are real and distinct,
there exists a matrix $T\in\mathbb{R}^{s\times s}$ such that
\begin{align}
T^{-1} A T = \diag (\lambda_1 , \dots, \lambda_s ), \qquad
\lambda_1,\dots,\lambda_s \in \mathbb{R}.
\end{align} 
Let $M = \paren*{I_{sN} - h A \otimes J_0 }$
and define $\overline{M}$ by $\overline{M} = (T^{-1}\otimes I_N )M (T\otimes I_N )$.
Then it follows that
\begin{align}
\overline{M} 
= \diag ( I_N - h\lambda_1 J_0, \dots,  I_N - h\lambda_s J_0),
\end{align}
where the right hand side expresses a block diagonal matrix.
Therefore, the formula \eqref{sNewton} is translated into
\begin{align*}
Y^{l+1} = Y^l + \rho^l , \qquad l =0,1,2,\dots ,
\end{align*}
where $\rho^l$ is calculated based on the following relation
\begin{align}
& \overline{\Phi} (Y^l) = (T^{-1}\otimes I_N ) \Phi (Y^l), \label{ad1} \\
& \overline{M} \overline{\rho}^l = - \overline{\Phi} (Y^l), \label{tsNewton} \\
& \overline{\rho}^l = (T^{-1}\otimes I_N ) \rho^l. \label{ad2}
\end{align}
The key point of this idea is that
the linear system \eqref{tsNewton} of size $sN$ consists of $s$ linear systems of size $N$.

Below we consider to what extent the computational cost is reduced.
For both approaches, we equally pay $N^2$ operations to compute $J_0$.

\begin{itemize}
\item Standard iteration based on \eqref{sNewton}.\\
We compute the $LU$ factors of $(I_{sN}-hA\otimes J_0)$,
where a small constant times $(sN)^3$ operations are needed.
Then we solve \eqref{sNewton} iteratively until convergence,
where $\mathcal{O}(s^2N^2)$ operations are needed per each iteration.
\begin{remark}
Note that the cost of the LU factorisation can be further reduced
by making use of complex conjugate eigenvalues pairs~\cite{hw96ii}.
For example, the cost of the LU factorisation for the 2-stage Gauss method
can be reduced from $\mathcal{O}(8N^3)$ to $\mathcal{O}(4N^3)$.
\end{remark}
\item Iteration based on \eqref{tsNewton}.\\
We compute the $LU$ factors of $(I_{N}-h\lambda_1 J_0),\dots,
(I_{N}-h\lambda_s J_0)$,
where only $\mathcal{O}(sN^3)$ operations are needed.
They can be reduced to $\mathcal{O}(N^3)$ in a parallel architecture.
Next we solve \eqref{tsNewton} iteratively until convergence,
where only $\mathcal{O}(sN^2)$ operations are needed,
and they can be further reduced to $\mathcal{O}(N^2)$ if parallelism is available.
Note that although we compute  \eqref{ad1} and \eqref{ad2} additionally,
the computational costs of these parts are just $\mathcal{O}(N)$, and thus less important.
\end{itemize}

\subsection{Efficient continuous stage Runge--Kutta methods and their implementation}
We apply the concept of the previous subsection for RK methods with real eigenvalues to CSRK methods.

First, let us express $Y_\tau$ in \eqref{CSRK1} as
\begin{align}
Y_\tau = y_0 l_0 (\tau) + \sum_{i=1}^s Y_{c_i} l_i(\tau)
\end{align}
where $l_i(\tau)$ is defined by
\begin{align}\label{defl}
l_i (\tau) = \prod_{j=0,\ j\neq i}^s \frac{\tau-c_j}{c_i-c_j}
\qquad
(c_0=0),
\qquad
i=0,1,\dots,s.
\end{align}
We regard \eqref{CSRK1} as
a system of nonlinear equations
in terms of $Y_{c_1},\dots,Y_{c_s}$:	
\begin{align} \label{ns1:CSRK}
Y_{c_i}= y_0 + h\int_0^1 A_{c_i,\zeta} f(Y_\zeta)\,\rmd\zeta, \qquad i=1,\dots,s.
\end{align}
Here we merely evaluate \eqref{CSRK1} at the nodes $\tau=c_1,\dots,c_s$.
As this was done for RK methods, the expression \eqref{ns1:CSRK}
is further re-expressed in a form
\begin{align}\label{ns2:CSRK}
\Phi (Y) = Y-e_s\otimes y_0 - h
\begin{bmatrix}
\displaystyle{\int_0^1 A_{c_1,\zeta} f(Y_\zeta)\,\rmd \zeta} \\
\vdots \\
\displaystyle{\int_0^1 A_{c_s,\zeta} f(Y_\zeta)\,\rmd \zeta}
\end{bmatrix}
=0,
\end{align}
where
\begin{align}
Y = 
\begin{bmatrix}
Y_{c_1} \\ \vdots \\ Y_{c_s}
\end{bmatrix}
\in\mathbb{R}^{sN}.
\end{align}
Since the derivation of $\Phi$ with respect to $Y$ becomes
\begin{align}
\Phi^\prime (Y)
=
I_{sN} - h
\begin{bmatrix}
\displaystyle{\int_0^1 A_{c_1,\zeta} J(Y_\zeta) l_1 (\zeta)\,\rmd\zeta}
& \cdots &
\displaystyle{\int_0^1 A_{c_1,\zeta} J(Y_\zeta) l_s (\zeta)\,\rmd\zeta}
\\
\vdots & \ddots & \vdots
\\
\displaystyle{\int_0^1 A_{c_s,\zeta} J(Y_\zeta) l_1 (\zeta)\,\rmd\zeta}
& \cdots &
\displaystyle{\int_0^1 A_{c_s,\zeta} J(Y_\zeta) l_s (\zeta)\,\rmd\zeta}
\end{bmatrix},
\end{align}
the application of the simplified Newton method for solving \eqref{ns2:CSRK}
gives the iteration formula
\begin{align}
\paren*{I_{sN} - h E \otimes J } \rho^l = - \Phi (Y^l), \qquad
Y^{l+1} = Y^l + \rho^l , \qquad l =0,1,\dots ,
\end{align}
where the matrix $E\in\mathbb{R}^{s\times s}$ is defined by
\begin{align}\label{matE}
E_{ij} =  \int_0^1 A_{c_i,\zeta} l_j(\zeta)\,\rmd\zeta.
\end{align}
Therefore, if the matrix $E$ of a CSRK method has only real, distinct eigenvalues,
the CSRK method can be implemented efficiently.
The following theorem indicates that one does not have to be
concerned about the $c_i$ values
when evaluating the eigenvalues of $E$.

\begin{theorem}\label{th:eiE}
The eigenvalues of the matrix $E$ defined in \eqref{matE} are
independent of the $c_i$ values.
\end{theorem}

\begin{proof}
We define a Vandermonde matrix
\begin{align*}\def\arraystretch{1.3}
V = \begin{bmatrix}
c_1 & c_1^2 & \cdots & c_1^s \\
c_2 & c_2^2 & \cdots & c_2^s \\
\vdots & \vdots & \ddots & \vdots \\
c_s & c_s^2 & \cdots & c_s^s 
\end{bmatrix}.
\end{align*}
Note that the matrix $V$ is non-singular, if the $c_i$ values are distinct. 

Below, we show that $V^{-1}EV$ is independent of the $c_i$ values.
We express $A_{\tau,\zeta}$ as $A_{\tau,\zeta} = \sum_{i=1}^s \tau^i P_i (\zeta)$
where $P_1,\dots,P_s$ are polynomials in $\zeta$.
Since
$
A_{c_i,\zeta} = \sum_{k=1}^s c_i^k P_k(\zeta)
$,
the matrix $E$ is expressed as
\begin{align*}
E_{ij} = \int_0^1 \sum_{k=1}^s c_i^k  P_k(\zeta) l _j (\zeta)\,\rmd \zeta 
=
\sum_{k=1}^s V_{ik} \int_0^1  P_k(\zeta) l _j (\zeta)\,\rmd \zeta .
\end{align*}
Therefore, it follows that
\begin{align}
(V^{-1}EV)_{ij} 
=\sum_{k=1}^s \int_0^1 P_i(\zeta) l_k(\zeta)\,\rmd \zeta
\cdot c_k^j
=\int_0^1 P_i (\zeta) \zeta^j \,\rmd \zeta.
\label{eq:TET}
\end{align}
In the second equality, the relation
$\sum_{k=1}^s c_k^j l _k(\zeta)= \zeta^j$
is used.
\end{proof}

The expression of $V^{-1}EV$ \eqref{eq:TET} can be
further simplified.
Since 
\begin{align*}
P_i(\zeta) = \frac{1}{i} 
\begin{bmatrix}
M_{i1} & \dots & M_{is}
\end{bmatrix}
\begin{bmatrix}
1 \\ \zeta \\ \vdots \\ \zeta ^{s-1}
\end{bmatrix}
\end{align*}
where the matrix $M$ is defined in \eqref{csrkm}, 
the component of $V^{-1}EV$ can be written as
\begin{align*}
(V^{-1}EV)_{ij} &= \int_0^1 P_i(\zeta)\zeta^j \,\rmd\zeta
=
\frac{1}{i}\int_0^1 
\begin{bmatrix}
M_{i1} & \dots & M_{is}
\end{bmatrix}
\begin{bmatrix}
\zeta^j \\  \vdots \\ \zeta ^{j+s-1}
\end{bmatrix}\,\rmd\zeta \\
&=
\frac{1}{i}
\paren*{\frac{M_{i1}}{j+1} + \dots + \frac{M_{is}}{j+s}}.
\end{align*}
Hence, we have the following formula
\begin{align}\def\arraystretch{1.3}
V^{-1}EV = 
\begin{bmatrix}
1 & & & \\ & \frac{1}{2} & & \\ & & \ddots & \\ & & & \frac{1}{s}
\end{bmatrix}
M
\begin{bmatrix}
\frac{1}{2} & \frac{1}{3} & \cdots & \frac{1}{s+1} \\ 
\frac{1}{3} & \frac{1}{4} & \cdots  & \frac{1}{s+2} \\ 
\vdots & \vdots & \ddots & \vdots \\
\frac{1}{s+1}& \frac{1}{s+2} & \cdots & \frac{1}{2s}
\end{bmatrix}.
\label{VEV}
\end{align}

In the following subsections, we construct efficient fourth and sixth order
energy-preserving integrators 
by making use of the idea presented in this subsection.

\subsection{Three-degree fourth order integrators}
\subsubsection{Derivation}
We derive fourth order energy-preserving integrators.
Since the only fourth order CSRK method with degree two
is the AVF method,
we set the degree of CSRK methods to $s=3$.
We also assume that $B_\zeta = 1$.
According to \autoref{th:order2k}, 
a CSRK method is of order four
if the symmetric matrix $M$ satisfies
\begin{align}\def\arraystretch{1.3}
\begin{bmatrix}
1 & \frac{1}{2} & \frac{1}{3} \\
\frac{1}{2} & \frac{1}{3} & \frac{1}{4} \\
\frac{1}{3} & \frac{1}{4} & \alpha
\end{bmatrix}
M =
\begin{bmatrix}
1 & & \\ & 1 & \\ & & 1
\end{bmatrix}
\end{align}
with a parameter $\alpha \in \bbR$.
When $\alpha \neq 7/36$, $M$ is analytically calculated to be
\begin{align*}\def\arraystretch{1.3}
M = \begin{bmatrix}
\alpha_1 + 4 & -6 \alpha_1 - 6 & 6\alpha_1\\
-6 \alpha_1 - 6 & 36\alpha_1+12 & -36\alpha_1 \\
6\alpha_1  &  -36 \alpha_1 & 36\alpha_1
\end{bmatrix},\qquad
\alpha_1 = \frac{1}{36\alpha -7}.
\end{align*}
The characteristic polynomial of the matrix \eqref{VEV}:
\begin{align*}\def\arraystretch{1.3}
V^{-1}EV = \begin{bmatrix}
1 & & \\ & \frac12 & \\ & & \frac13
\end{bmatrix}
M
\begin{bmatrix}
\frac12 & \frac13 & \frac14 \\
\frac13 & \frac14 & \frac15 \\
\frac14 & \frac15 & \frac16 
\end{bmatrix}
\end{align*}
is found to be
\begin{align*}
\varphi (\lambda )= \lambda^3 - \frac{1}{2}\lambda^2 + \paren*{\frac{1}{12}+\frac{\alpha_1}{300}}
\lambda - \frac{\alpha_1}{600}.
\end{align*}
It follows that $\varphi (\lambda )=0$ has three real, distinct roots
if
\begin{align*}
-\frac{\alpha_1}{300} > \frac{1}{6}2^{2/3} + \frac{5}{24}2^{1/3} + \frac{1}{4}
\approx
0.7770503941.
\end{align*}
Calculating the range of $\alpha$ is easy, and 
exactly the same evaluation was done by Butcher--Imran~\cite{bu13}
in the context of symplectic methods.

\subsubsection{Long time behaviour}
Since the new method surely preserves the energy,
one may expect a good long-time behaviour.
As a criterion of the long-time behaviour,
we here consider the conjugate-symplecticity~\cite{ha06}.

A one-step method $\Phi_h$ of order $p$ is said to be
conjugate-symplectic up to order $p+r$ ($r\geq 0$), if
there exists such
a change of coordinates $z=\chi (y)$
that is $\calO (h^p)$ close to the identity and
$\Psi_h = \chi \circ \Phi_h \circ \chi^{-1}$ satisfies
\begin{align*}
\Psi_h^\prime (z)^\top J \Psi_h^\prime (z) = J + \calO (h^{p+r+1}).
\end{align*}
The modified equation of such a method is Hamiltonian
up to the terms of size $\calO(h^{p+r})$,
and thus the behaviour of the method is like that of a symplectic integrator
on intervals of length $\calO(h^{-r})$.
See~\cite[Chapter~VI.8]{ha06} for more details.

In~\cite{ha10,ha13}, it is proved that
the AVF collocation method of order $4$ ($2s$ in general)
is conjugate-symplectic up to order $6$ ($2s+2$),
but not conjugate-symplectic up to a higher-order.
It follows from Theorem~5.11 in Hairer--Zbinden~\cite{ha13} that
the new method (of order $4$) is conjugate-symplectic up to order $6$
but not conjugate-symplectic up to a higher-order, independently of the choice of $\alpha$.
Therefore, one could expect that
the new fourth-order method behaves like the fourth-order AVF collocation method
in terms of conjugate-symplecticity.

\subsubsection{Efficiency}
Although the computational cost in each time step of the derived integrators
is reduced to about one-eighth
(or one-fourth if the AVF collocation method is implemented by making use of 
complex conjugate eigenvalues pairs),
the actual local truncation error of the new integrators
is worse than the AVF collocation method.
Therefore,
we here carefully consider the efficiency of the proposed integrators
by taking both truncation errors and computational costs into consideration.
Below, we denote $-\alpha_1/300$ by $\theta$.

\autoref{table:cb5} shows coefficients of the elementary differentials with trees of order 5
for the B-series expansions of the exact solution, the two-degree fourth-order AVF collocation method
and the proposed method.
It is estimated that when we use the same stepsize the error of the proposed method
is approximately $(60\theta+1)$ times bigger than that of the fourth-order
AVF collocation method.
On the other hand, for large problems,
the computation of the proposed method is 8 (or 4) times faster than 
the AVF collocation method
per each time step.
Thus it is fair to
compare the local errors of the AVF collocation method with the stepsize $8h$
and the proposed method with stepsize $h$ at $t=t_0+8h$
so that the overall computational costs are almost the same.
The former is denoted by $err_{\text{AVF}}$ and the latter by $err_{\text{New}}$.
Then the ratio of the errors at $t=t_0+8h$  is roughly estimated to be
\begin{align}
\frac{8\cdot err_{\text{New}}}{ err_{\text{AVF}}} 
\approx \frac{8\cdot (60\theta+1) h^5}{(8h)^5} = 
\frac{60\theta+1}{4096}.
\end{align}
Therefore, the proposed method is more efficient than the fourth-order AVF collocation method
(in the sense that the above ratio is less than $1$)
if 
\begin{align*}
0.7770503941 < \theta < \frac{273}{4} = 68.25.
\end{align*}
Note that if we consider the new method is only 4 times faster than the 
AVF collocation method,
this range of $\theta$ changes to
\begin{align*}
0.7770503941 < \theta < \frac{17}{4} = 4.25,
\end{align*}
and thus the proposed method is still efficient.

\begin{table}[htbp]
\caption{Coefficients of the elementary differentials with trees of order 5
for the B-series expansions of the exact solution, the fourth-order AVF collocation method
and the proposed method.}
\label{table:cb5}
\centering
{\tiny
{\renewcommand\arraystretch{2}
\begin{tabular}{cccccccccc}
\hline 
 $t$ & \tix & \tx  & \txi 
 & \txii & \txiii & \txiv 
 & \txv & \txvi & \txvii  \\ 
exact solution & $\frac{1}{5}$ & $\frac{1}{10}$ & $\frac{1}{15}$ 
& $\frac{1}{30}$ & $\frac{1}{20}$ & $\frac{1}{20}$ 
&$\frac{1}{40}$ &$\frac{1}{60}$ &$\frac{1}{120}$ \\ 
\hline
AVF collocation & $\frac{1}{5}$ & $\frac{1}{10}$ & $\frac{5}{72}$ 
& $\frac{5}{144}$ & $\frac{1}{20}$ & $\frac{1}{20}$ 
&$\frac{1}{40}$ &$\frac{1}{72}$ &$\frac{1}{144}$ \\
AVF collocation $-$ exact & $0$ & $0$ & $\frac{1}{360}$ 
& $\frac{1}{720}$ & $0$ & $0$ 
&$0$ &$-\frac{1}{360}$ &$-\frac{1}{720}$ \\
\hline
Proposed method & $\frac{1}{5}$ & $\frac{1}{10}$ & $\frac{12\theta+5}{72}$ 
& $\frac{12\theta+5}{144}$ & $\frac{1}{20}$ & $\frac{1}{20}$ 
&$\frac{1}{40}$ &$\frac{1-12\theta}{72}$ &$\frac{1-12\theta}{144}$ \\
Proposed method $-$ exact & $0$ & $0$ & $\frac{60\theta+1}{360}$ 
& $\frac{60\theta+1}{720}$ & $0$ & $0$ 
&$0$ &$-\frac{60\theta+1}{360}$ &$-\frac{60\theta+1}{720}$ \\
\hline
\end{tabular} 
}}
\end{table}

\subsection{Five-degree Sixth order integrators}
We next consider the derivation of five-degree sixth order integrators.
Since it turned out that
intended integrators cannot be constructed with $s=4$,
we here set $s=5$.
We also assume that $B_\zeta = 1$.
A CSRK method is of order at least six
if the symmetric matrix $M$ satisfies
\begin{align}\def\arraystretch{1.3}
\begin{bmatrix}
1 & \frac{1}{2} & \frac{1}{3} & \frac{1}{4} & \frac{1}{5} \\
\frac{1}{2} & \frac{1}{3} & \frac{1}{4} & \frac{1}{5} & \frac{1}{6} \\
\frac{1}{3} & \frac{1}{4} & \frac{1}{5} & \frac{1}{6} & \frac{1}{7} \\
\frac{1}{4} & \frac{1}{5} & \frac{1}{6} & \alpha & \beta \\
\frac{1}{5} & \frac{1}{6} & \frac{1}{7} & \beta & \gamma
\end{bmatrix}
M =
\begin{bmatrix}
1 & & & &\\ & 1 & & & \\ & & 1 & & \\ & & & 1 & \\ & & & & 1
\end{bmatrix}
\end{align}
with parameters $\alpha,\beta,\gamma \in \bbR$.
Recall that in the derivation of the three-degree fourth order integrators,
it was easy to evaluate eigenvalues of the matrix \eqref{VEV},
because there is only one free parameter.
However, we find it increasingly difficult to do a similar estimate
for sixth order integrators:
the application of the Sturm theorem is too complicated 
due to three free parameters.
Here, we only show that there are choices of parameters such that 
the matrix \eqref{VEV} has only real, distinct eigenvalues,
by giving a concrete choice.

It is difficult to express the matrix $M$
analytically in terms of the parameters $\alpha$, $\beta$ and $\gamma$.
Thus, instead of requiring all eigenvalues of \eqref{VEV} to be real,
we require real eigenvalues of 
\begin{align*}\def\arraystretch{1.3} N=
\begin{bmatrix}
1 & & & &  \\ & 2 & & & \\ & & 3 &  & \\ & & & 4 & \\ & & & & 5
\end{bmatrix}
\begin{bmatrix}
\frac{1}{2} & \frac{1}{3} & \frac{1}{4}& \frac{1}{5} & \frac{1}{6} \\ 
\frac{1}{3} & \frac{1}{4} & \frac{1}{5}& \frac{1}{6} & \frac{1}{7} \\ 
\frac{1}{4} & \frac{1}{5} & \frac{1}{6}& \frac{1}{7} & \frac{1}{8} \\ 
\frac{1}{5} & \frac{1}{6} & \frac{1}{7}& \frac{1}{8} & \frac{1}{9} \\ 
\frac{1}{6} & \frac{1}{7} & \frac{1}{8}& \frac{1}{9} & \frac{1}{10} 
\end{bmatrix}^{-1}M^{-1}.
\end{align*}

We write 
\begin{align*}
\alpha = \frac{1}{7}+\alpha^\ast, \qquad \beta = \frac{1}{8} + \beta^\ast, \qquad
\gamma = \frac{1}{9} + \gamma^\ast.
\end{align*}
Let $\varphi (\lambda) = \det (\lambda I - N )$.
It is possible to find a parametric solution to the problem
$\varphi (0) = \varphi^\prime (0) = 0$
in terms of $\alpha^\ast$, $\beta^\ast$ and $\gamma^\ast$.
One way of doing this is explained in Step 1 below.
Here are three steps that need to be carried out to find a real root solution to $\varphi$.

\begin{enumerate}
\item
We choose the following parameters
\begin{align*}
\alpha^\ast = -\frac{1+7a^2}{2800}, \qquad
\beta^\ast = -\frac{3+ac}{4200}, \qquad
\gamma^\ast = -\frac{64+c^2}{44100}
\end{align*}
where $a$ and $c$ are real or purely imaginary numbers,
so that the first and second degree terms in $\varphi(\lambda)$ are both 0.
\item
It is found that
\begin{align*}
\varphi(\lambda) = \lambda^2  \paren*{\lambda^3 + (\delta-2)\lambda^2
-6\delta\lambda + 12\delta}
\end{align*}
where $\delta = 18(21a-c)^2-10$.
This polynomial has three real roots for $\delta \in [0,0.2144863035]$
where the upper bound is the real zero of 
\begin{align*}
\delta^3 + 24\delta^2 + 144\delta -32 = 0.
\end{align*}
We choose the parameters $a$ and $c$ so that $\delta \in [0,0.2144863035]$.
A possible choice is $(a,c,\delta) = (\frac{1}{12},1,\frac{1}{8})$.
\item
We add small perturbation to $\alpha$ or $\beta$ or $\gamma$
in the correct direction so that the two zero roots remain real.
For example, reducing $\gamma$ by $10^{-10}$ gives the real roots
\begin{align*}
-0.8831274189, \ -0.06591290801, \ 0.06591290335, \ 0.9185056830,\  1.839542361
\end{align*}
\end{enumerate}

\begin{remark}
Although 
the five-degree sixth order integrator has been derived above,
it is still difficult to give a necessary and sufficient condition for 
the parameters $\alpha$, $\beta$ and $\gamma$
such that the matrix \eqref{VEV} has only real eigenvalues.
Therefore, we here leave the discussion about the efficiency untouched.
\end{remark}

\section{Concluding remarks}
\label{sec6}
In this paper,
we have shown the energy-preserving condition of CSRK methods for Hamiltonian systems,
characterized the order conditions and parallelizable condition in terms of the matrix $M$.
As applications, we have derived new fourth and sixth order integrators
which can be implemented more efficiently 
than the standard method, i.e. the AVF collocation method.
In the derivation, the eigenvalues of the matrix $E$ played an important role.

We note several directions for future work.
From a theoretical viewpoint, it is hoped that the conjecture discussed in \autoref{rem:con}
will be solved.
Considering other applications of the energy-preserving condition of CSRK methods
would be interesting.
We are currently trying to give a more systematic approach to 
deriving higher order efficient energy-preserving integrators.

\section*{Acknowledgments}
The authors are grateful for various comments by anonymous referees.
The first author was supported by the Research Fellowship of the Japan Society for the Promotion of Science for Young Scientists.
The second author was supported by the Marsden Fund of New Zealand.

\end{document}